%% file: Main_TimeWeights_Blackstock.tex
\documentclass[11pt, twoside, leqno]{amsart}  
\input{preamble}     
\makeatletter
\@namedef{subjclassname@2020}{\textup{2020} Mathematics Subject Classification}
\makeatother

\title[Time-weighted estimates for the Blackstock equation]{Time-weighted estimates for the Blackstock equation in nonlinear ultrasonics}       
         
\author{Vanja Nikoli\'c$^\dagger$}  
\thanks{$^\dagger$Department of Mathematics,
	Radboud University,      
	Heyendaalseweg 135,    
	6525 AJ Nijmegen, The Netherlands (\href{vanja.nikolic@ru.nl}{vanja.nikolic@ru.nl})}   
\author{Belkacem Said-Houari$^\ddag$}
\thanks{$^\ddag$Department of Mathematics, College of Sciences, University of
	Sharjah, P.\ O.\ Box: 27272, Sharjah, United Arab Emirates (\href{bhouari@sharjah.ac.ae}{bhouari@sharjah.ac.ae})}
\begin{document}
\vspace*{8mm}  
\begin{abstract}
High frequencies at which ultrasonic waves travel give rise to nonlinear phenomena. In thermoviscous fluids, these are captured by Blackstock's acoustic wave equation with strong damping. We revisit in this work its well-posedness analysis. By exploiting the parabolic-like character of this equation due to strong dissipation, we construct a time-weighted energy framework for investigating its  local solvability. In this manner, we obtain the small-data well-posedness on bounded domains under less restrictive regularity assumptions on the initial conditions compared to the known results. Furthermore, we prove that such initial boundary-value problems for the Blackstock equation are globally solvable and that their solution decays exponentially fast to the steady state. 
\end{abstract}   
	\vspace*{-7mm}  
	\maketitle               
 \section{Introduction} \label{Sec:ProblemSetting}
Blackstock's wave equation arises as a model of nonlinear propagation of ultrasonic waves through thermoviscous fluids, alternative to the Kuznetsov equation~\cite{kuznetsov1971equations}. Originally derived by Blackstock in~\cite{blackstock1963approximate}, it later appeared independently in the works of Crighton~\cite{crighton1979model} and Lesser and Seebass~\cite{lesser1968structure}.  It is expressed in terms of the acoustic velocity potential $\psi=\psi(x,t)$ by
\begin{subequations}\label{Main_Problem}
	\begin{equation} \label{B_eq}
	 \begin{aligned}
	\psi_{tt}-c^2(1-2k \psi_t)\Delta \psi-b \Delta \psi_t + 2 \sigma \nabla \psi \cdot \nabla \psi_t=0. 
	\end{aligned}   
	\end{equation}  
Here $c>0$ is the speed of sound in the fluid, $b>0$ the sound diffusivity, and $k$, $\sigma \in \R$ nonlinear coefficients. Equation~\eqref{B_eq} can be seen as an approximation of the compressible Navier--Stokes--Fourier system of governing equations of nonlinear sound motion.  It was demonstrated in~\cite{christov2016acoustic} that, in the small Mach number limit, the 1D Blackstock equation shows good agreement with the exact governing system based on the fully nonlinear theory. In the lossless case ($b=0$), a comparison of different weakly nonlinear acoustic models performed in~\cite{christov2007modeling} singles out the Blackstock equation as the most consistent one. \\
\indent  The well-posedness  and regularity analysis of nonlinear acoustic wave equations has gained a lot of interest in recent years; see~\cite{kaltenbacher2009global, kaltenbacher2011well, meyer2011optimal, meyer2012global, kaltenbacher2022limiting, bongarti2021vanishing, tani2017math} for a selection of relevant results  as well as the review paper~\cite{kaltenbacher2015mathematics}. One of the challenges in the well-posedness analysis of such models remains their solvability under reduced assumptions on data in terms of their smoothness and size. \\
\indent In this work we consider Blackstock's equation on smooth bounded domains $\Omega \subset \R^d$, where $d \in \{1,2,3\}$, and couple it with boundary and initial conditions:
\begin{equation} \label{coupled_problem_IC}
\psi \vert_{\partial \Omega}=0, \qquad (\psi, \psi_t)\vert_{t=0}= (\psi_0, \psi_1).   
\end{equation}   
\end{subequations}  
A natural question arises: \emph{What is the minimal regularity of initial data is that ensures (at least) local existence and uniqueness of the solution to \eqref{Main_Problem}}? In answering this question, the aim of this work is threefold. First, we prove  
 a large-time existence and uniqueness result  in general three-dimensional domains for small data in 
\begin{equation}\label{H_2_Initial}
(\psi_0, \psi_1)\in H^2(\Omega) \cap \Honezero  \times \Honezero,
\end{equation}
 thereby improving upon the existing results in the literature which assume at least
 \begin{equation}\label{H_3_Intial}
(\psi_0, \psi_1) \in H^3(\Omega) \cap H_0^1(\Omega)  \cap H^2(\Omega)  \times \Honezero;
\end{equation}
see~\cite{fritz2018well, kawashima1992global}.  To this end, we exploit the strong damping present in the equation (with $b>0$) which contributes to its parabolic character. The parabolic nature of the problem will allow us to devise suitable (time-weighted) energy estimates under minimal regularity assumptions on the initial conditions.  \\
 \indent Secondly, we address the question of existence of a global solution for small initial data satisfying \eqref{H_2_Initial}. The proof is conducted by developing an energy method to arrive at suitable uniform estimates with respect to time for the solution of \eqref{Main_Problem}, and thus extend a local solution to be global. Thirdly, we prove the asymptotic stability as $t\rightarrow \infty$  of the solution. More precisely, we show that the solution decays to the steady state  with an exponential  decay rate. \\
\indent The time-weighted energy method has been successfully used for problems related to the heat equation~\cite{danchin2013new} and the Navier--Stokes equations~\cite{LI_2017,Danchin_Much_2019,Paicu_Zhang_2013}, where it allows gaining more regularity with minimal assumptions on the initial data.  Time-weighted estimates have also been employed in the numerical analysis of strongly damped linear wave equations in~ \cite{larsson1991finite}.  Inspired by \cite{Danchin_Much_2019} and by exploiting  the parabolic nature of \eqref{Main_Problem} with $b>0$, we use a maximal regularity estimate  for a linearized problem combined with the time-weighted energy method  to extract higher regularity of the solution under the minimal assumption \eqref{H_2_Initial} on the initial data. More precisely, we prove that for any fixed final propagation time $0<T<\infty$ and for all $t\in (0,T)$, the solution $\psi$ satisfies
\begin{equation} \label{Regularity_time_Weight}
\begin{aligned}
&\sqrt{t}\psi_{tt} \in L^\infty(0,T; \Ltwo), \ \sqrt{t} \nabla \psi_{tt} \in L^2(0,T; \Ltwo), \\
&\, \sqrt{t} \D \psi_t \in L^\infty(0,T; L^2(\Omega));
\end{aligned}
\end{equation}
see Theorem~\ref{Thm:LocalWellP} for details. Without the time weight, regularity \eqref{Regularity_time_Weight} would follow by an energy method only under additional smoothness assumption on the data. One of the key ideas in proving \eqref{Regularity_time_Weight} is to write a linearization of \eqref{Main_Problem} as a nonlocal heat equation for $v=\psi_t$. The presence of the nonlocal term $\Delta \psi$ in \eqref{v_Equation} makes the analysis more involved.  The analysis of a linearization is then combined with Banach's fixed-point theorem to arrive at the well-posedness of the nonlinear problem with small enough data, and arbitrary large final time $T\in (0,\infty)$. \\
\indent  Although this result guarantees existence and uniqueness of the solution in very regular spaces and there is no restriction on the time of existence $T$, we cannot take $T=\infty$ since  the estimates are time dependent. To obtain the estimates uniform in time and prove eventually the global existence (i.e., $T=\infty$), we apply a new method based on the construction of suitable compensating functions that encode the dissipation property of \eqref{Main_Problem}.  More precisely,
by restricting the regularity to the energy space and using a remarkably simple energy method performed directly on the nonlinear problem \eqref{Main_Problem}, we also show that for small initial data, the solution is global in time and decays to the steady state exponentially fast; see Theorem~\ref{Thm:GlobalWellP} below for details. It is important to note that the smallness assumption  on the initial data seems necessary since solution for large initial data may blow up in finite time.  \\
\indent We note that we expect that the time-weighted energy framework developed in this work can be extended to more general (mixed) boundary conditions and that the ideas put forward here can be transferred to some extent to the study of suitable numerical discretizations of strongly damped nonlinear wave equations as well. We mention in passing that the local well-posedness of this problem in the hyperbolic case ($b=0$) follows by~\cite[Theorem 5.1]{kaltenbacher2022limiting}, where \eqref{B_eq} is obtained in the limit of a fractionally damped wave equation for the vanishing sound diffusivity. \\
\indent The rest of the paper is organized as follows.  We begin in Section~\ref{Sec:Preliminaries} by recalling useful interpolation inequalities that we often employ in the analysis.  In Section~\ref{Sec:LinProblem} we devise time-wieghted estimates for a linearization of \eqref{B_eq}.  Section~\ref{Sec:NonLinProblem}  is dedicated to the analysis of the nonlinear problem which relies on a fixed-point argument under the assumption of small enough initial data. We conclude in \ref{Sec:Global} with investigation of the global solvability of the problem. Our main results are contained in Theorems~\ref{Thm:LocalWellP} and~\ref{Thm:GlobalWellP}.
\section{Theoretical preliminaries}  \label{Sec:Preliminaries} 
In this section, we collect certain helpful embedding results and inequalities that we will repeatedly use in the proofs. Throughout the paper, we assume that $\Omega \subset \R^d$, where $d \in \{1,2,3\}$, is a bounded and $C^{1,1}$ regular or polygonal/polyhedral and convex domain. We denote by $T>0$ the final propagation time. We make the following assumptions on the involved coefficients:
\begin{equation} \label{assumptions_coefficients}
c>0, \qquad b>0, \qquad k, \sigma \in \R.
\end{equation}
\subsection*{Notation} Below we write $x \lesssim y$ to denote $x \leq Cy$ where $C$ is a generic positive constant that does not depend on $T$.  We write $\lesssim_T$ when the hidden constant depends on $T$ in such a manner that it tends to $+\infty$ as $T \rightarrow + \infty$. We often omit the spatial and temporal domain when writing norms; for example, $\|\cdot\|_{L^p(L^q)}$ denotes the norm in $L^p(0,T; L^q(\Omega))$. \\[2mm]
\noindent In upcoming proofs, we will often use the continuous embeddings~\cite[Theorem 5.4]{adams2003sobolev} 
\begin{equation}  
	\begin{aligned}   
		W^{k,p}(\Omega)\hookrightarrow L^q(\Omega),\quad & p\leq q<\infty\quad&& \text{if}\quad d\leq kp\\
		W^{k,p}(\Omega)\hookrightarrow L^q(\Omega),\quad & p\leq q\leq \frac{dp}{d-kp}\quad&& \text{if}\quad d> kp. 
	\end{aligned}
\end{equation}
We will also rely on the following application of H\"older's inequality:
\begin{equation} \label{Holder_Bochnerspaces}
	\begin{aligned}
		\|u\|_{L^q(0,T; L^p(\Omega))} \leq \|u\|_{L^{q_1}(0,T; L^{p_1}(\Omega))}^{1-\gamma} \|u\|_{L^{q_2}(0,T; L^{p_2}(\Omega))}^\gamma, \quad \gamma \in [0,1], 
	\end{aligned}
\end{equation}
with integers $p, q, q_{1,2}, p_{1,2} \in [1, \infty]$, such that
\[ \frac{1}{q}=\frac{1-\gamma}{q_1}+\frac{\gamma}{q_2}, \qquad  \frac{1}{p}=\frac{1-\gamma}{p_1}+\frac{\gamma}{p_2}.\]   
\subsection*{Interpolation inequalities} We will also need Agmon's interpolation inequality~\cite[Ch.\ 13]{agmon2010lectures} for functions in $H^2(\Omega)$:
\begin{equation}\label{Agmon_Inequality}
\| u\|_{L^\infty(\Omega)} \leq C_{\textup{A}} \|u\|_{H^2(\Omega)}^{d/4}\|u\|_{L^2(\Omega)}^{1-d/4}.
\end{equation}
Let $\alpha \in L^2(0,T; H^2(\Omega))$. Using Agmon's and H\"older's inequalities, it follows that
\begin{equation} \label{intp_ineq_LtwoLinf}
\begin{aligned}
\|\alpha\|_{L^2(0,T; L^\infty(\Omega))} \lesssim& \left\|\|\alpha(t) \|_{H^2(\Omega)}^{d/4}\right\|_{L^{2/(d/4)}(0,T)}\left\|\|\alpha(t)\|^{1-d/4}_{L^2(\Omega)} \right\|_{L^{2/(1-d/4)}(0,T)}\\
\lesssim&\, \|\alpha \|_{L^2(0,T;H^2(\Omega))}^{d/4}\|\alpha\|_{L^2(0,T;L^2(\Omega))}^{1-d/4}.
\end{aligned}
\end{equation}      
We also have the following helpful inequality. 
\begin{lemma}[see p.\ 74 in~\cite{Ladyzhenskaya_book}]\label{Lemma_Interpolation}
Let $q\in [2, \frac{2d}{d-2}]$ if $d>2$ and $2\leq q< \infty$ for $d=2$ and $2\leq q\leq \infty$ for $d=1$.  Let $u \in H^1(\Omega)$. Then
\begin{equation}
\|u\|_{L^q(\Omega)}\leq C\|u\|_{H^1(\Omega)}^{\frac{d}{2}-\frac{d}{q}} \|u\|_{L^2(\Omega)}^{1-\frac{d}{2}+\frac{d}{q}}, 
\end{equation}
 where $C$ is a constant which depends only on $d$ and $q$. 
\end{lemma}        
\noindent Particularly useful for the upcoming analysis will be cases $q=3$ and $q=4$:
\begin{equation} \label{intp_ineq_q_34}
\begin{aligned}
	\Vert u \Vert_{L^3(\Omega)}\lesssim&\, \Vert  u\Vert_{H^1(\Omega)}^{d/6} 
\Vert u \Vert_{ L^2(\Omega)}^{1-d/6}, \\	
	\Vert u \Vert_{L^4(\Omega)}\lesssim&\, \Vert  u\Vert_{H^1(\Omega)}^{d/4} 
	\Vert u \Vert_{ L^2(\Omega)}^{1-d/4}	.  
\end{aligned}	
\end{equation}
\subsection*{A generalization of Gronwall's inequality} Finally, we state the following result, which will be needed in the proof of the global solvability and exponential decay of the solution.  
\begin{lemma}[see Lemma 4.5 in~\cite{Mischler}]\label{Lemma_Expo}
Assume that $u\in C([0,\infty); \R_+)$ satisfies the following inequality 
\begin{equation}
u(t)\leq c_1e^{at}u(0)+c_2\int_0^t e^{a(t-s)} u(s)^{1+\kappa}\ds,\quad \forall t\geq 0,
\end{equation}
for some constants $c_1>1$, $c_2$, $\kappa>0$, and $a<0$. 
Then, under the smallness assumption 
\begin{equation}
a+(1+1/\kappa)c_22^{\kappa}c_1^\kappa u(0)^{\kappa}<0,
\end{equation}
it holds 
\begin{equation}
u(t)\leq \left(1+\frac{c_2c_1^{\kappa}u(0)^\kappa}{a \kappa+(1+\kappa)c_2 2^{\kappa}c_1^{\kappa}u(0)^{\kappa}}\right)c_1 e^{a t}u(0). 
\end{equation}
\end{lemma}
\section{Time-weighted estimates for a linearized problem} \label{Sec:LinProblem}
We first analyze a linearization of \eqref{B_eq} given by
\begin{equation} \label{coupled_problem_linearized}
	\left. \begin{aligned}
		& \psi_{tt}-c^2(1-2k \alpha(x,t))\Delta \psi - b \Delta \psi_t+ 2 \sigma \nabla \psi \cdot \nabla \alpha(x,t)=  \tilde{f}
	\end{aligned} \right.        
\end{equation} 
 supplemented by initial and boundary conditions \eqref{coupled_problem_IC}.  The results of this section will play a key role when applying the fixed-point argument to the nonlinear problem later in Section~\ref{Sec:NonLinProblem}. 
Indeed, the variable coefficient $\alpha=\alpha(x,t)$ in \eqref{coupled_problem_linearized} serves as a placeholder for the previous fixed-point iterate of $\psi_t$. \\
\indent To exploit the parabolic character of \eqref{coupled_problem_linearized} for $b>0$, we define a new unknown $v=\psi_t$ so that
\begin{equation}\label{u_v_Eq}
\psi(x,t)=\psi_0(x)+\int_0^t v(x, s) \ds. 
\end{equation}
Consequently, we recast  the linearization of \eqref{Main_Problem} as
\begin{equation}\label{v_Equation}
\left\{
\begin{aligned}     
&v_{t} - b \Delta v-c^2 \Delta \psi=\, -2kc^2 \alpha(x,t) \Delta \psi-2 \sigma \nabla \psi \cdot \nabla \alpha(x,t) \quad \text{in} \ \Omega \times (0,T), \\
&v|_{t=0}=\, \psi _1, \\  
&v=0 \quad  \text{on}\quad   \partial \Omega.
\end{aligned}
\right.
\end{equation} 
 We note that the estimates below can be made rigorous using a Faedo--Galerkin procedure with smooth approximations of the solution in space combined with uniform energy estimates and compactness arguments; see, e.g.~\cite[Ch.\ 7]{evans2010partial}. As this is by now a rather standard procedure also in the context of nonlinear acoustic models (see, e.g.,~\cite{kaltenbacher2022parabolic, fritz2018well}), we omit the semi-discretization details in this work and focus on the main energy arguments in the presentation below.

\subsection{Estimates for the nonlocal heat equation}  We derive first the bounds for the solution of 
\begin{equation}\label{v_Equation_0}
\left\{
\begin{aligned}
&v_{t} - b \Delta v- c^2\Delta \psi=\, f  \quad \text{in} \ \Omega \times (0,T), \\
&v|_{t=0}=\, \psi _1\\  
&v=0 \quad  \text{on}\ \partial \Omega,
\end{aligned}
\right.
\end{equation}
where we have in mind that $f$ serves as a placeholder for
\begin{equation}\label{f_form_1}
f= -2kc^2 \alpha(x,t) \Delta \psi-2 \sigma \nabla \psi \cdot \nabla \alpha(x,t)
\end{equation}
and should be further estimated later on.

\begin{proposition}\label{Prop:Lin_dataH1}
	Given a final time $T>0$, let $f \in L^2(0,T; L^2(\Omega))$ and 
	\[
	(\psi_0, \psi_1) \in H_0^1(\Omega)  \times H_0^1(\Omega).
	\]
	Then the following estimate holds:
	\begin{equation}\label{E_1_Theta_1}
		\begin{aligned}
			\nLinfLtwo{ (v,\nabla \psi, \nabla v)}+\nLtwoLtwo{ (\nabla v, v_t)}	
			\lesssim 	\Vert  \psi_0 \Vert_{H^1}+\Vert  \psi_1 \Vert_{H^1} +\nLtwoLtwo{f}.
		\end{aligned}
	\end{equation}
	If additionally $\sqrt{t} f \in L^{2}(0,T; L^{2}(\Omega))$, then
	\begin{equation}\label{L_p_q_Estimates_Main}
	\begin{aligned}
		&\nLinfLtwo{ (v,\sqrt{ t}\nabla v)}+\nLtwoLtwo{ (\nabla v,\sqrt{ t} v_t)} \\
		\lesssim_T&\,  \| \psi_0\|_{H^1}+ \| \psi_1\|_{H^1}+\nLtwoLtwo{ f }+\nLtwoLtwo{ \sqrt{t} f}. 
	\end{aligned}  
\end{equation}
If $\sqrt{t} f_t \in L^2(0,T; H^{-1}(\Omega))$ as well, then
\begin{equation}\label{Ineq_weight_nabla_v}
\begin{aligned}
&\Vert \sqrt{t}v_t\Vert_{L^\infty (L^2)}^2 + \Vert \sqrt{t}\nabla v_t\Vert_{L^2(L^2)}^2 \\
\lesssim_T&\, \Vert  \psi_0 \Vert^2_{H^1}+\Vert \psi_1\Vert_{H^1}^2+\Vert f \Vert_{L^2(L^2)}^2+\Vert \sqrt{t}f_t\Vert_{L^2H^{-1}}^2.
\end{aligned} 
\end{equation}
\end{proposition}
\begin{proof}
By testing the heat equation in \eqref{v_Equation_0} by $v$, integrating by parts, and using $v=\psi_t$, we obtain
	\begin{equation}\label{Testing_v_1}
		\begin{aligned}
			\frac12	\frac{\textup{d}}{\dt}\Big(\|v\|^2_{L^2}+{c^2\|\nabla 
		\psi\|_{L^2}^2}\Big)+b\int_{\Omega}|\nabla v|^2\dx=\intO f v\dx. 
		\end{aligned}    
	\end{equation} 
Integrating \eqref{Testing_v_1} in time and using Young's $\varepsilon$-inequality together with Poincar\'e's inequality, yields 
	 \begin{equation}\label{E_0_Estimate_1}
		\begin{aligned}
			\| v(t)\|^2_{L^2}+\| \nabla \psi(t)\|^2_{L^2} +\int_0^t \|\nabla v\|_{L^2}^2 \ds  
			\lesssim\, 	\| \psi_1\|^2_{L^2}+\nLtwoLtwo{f}^2
		\end{aligned}
	\end{equation}
for all $t\in [0,T]$. Testing instead by $ v_t$ results in   
\begin{equation} \label{test_vt}
\frac{b}{2}\ddt\int_\Omega |\nabla v |^2\dx+\int_\Omega v_t^2\dx+c^2\ddt\int_\Omega \nabla \psi\cdot \nabla \psi_{t}\dx-c^2\|\nabla \psi_t\|_{L^2}^2=\int_\Omega  v _t f\dx.
\end{equation}
Integrating in time and using Young's inequality leads to 
\begin{equation} \label{test_vt_integrate}
\begin{aligned}
&\Vert \nabla v (t)\Vert_{L^2}^2+\int_0^t \Vert  v _t(s)\Vert_{L^2}^2\ds\\
\lesssim&\,\begin{multlined}[t]\Vert  \psi_0 \Vert_{H^1}^2+\Vert  \psi_1 \Vert_{H^1}^2 
+\int_0^t\Vert f(s)\Vert_{L^2}^2\ds
+ C(\varepsilon)\|\nabla\psi(t)\|_{L^2}^2\\+\varepsilon\|\nabla v(t)\|_{L^2}^2+\int_0^t\|\nabla v(s)\|_{L^2}^2\ds. \end{multlined}
\end{aligned}  
\end{equation}
By multiplying \eqref{E_0_Estimate_1} by $\lambda>0$, adding the result to \eqref{test_vt_integrate} and selecting $\varepsilon>0$ small enough and  $\lambda$ large enough, we obtain \eqref{E_1_Theta_1}. \\
\indent We prove estimate \eqref{L_p_q_Estimates_Main} next. To introduce the time weights, we multiply \eqref{test_vt} by $s \in (0,t)$, which leads to
\begin{equation}\label{Test_v_0_2}
\begin{aligned}
&\frac{b}{2} \dds \left(s \|\nabla v\|^2_{L^2}\right)+s\intO    v_{t}^2 \dx+c^2\dds \left(s\int_\Omega \nabla \psi\cdot \nabla \psi_{t}\dx\right)\\
=&\,\frac{b}{2}\|\nabla v\|^2_{L^2}+s\int_\Omega  v_t f\dx+c^2s\|\nabla v\|_{L^2}^2+c^2\int_\Omega \nabla \psi\cdot \nabla \psi_{t}\dx.
\end{aligned}
\end{equation}
Integrating the above equality over $s \in (0,t)$ for $t \in (0,T)$ yields
\begin{equation}\label{E_1_Int}
\begin{aligned}
&\frac{b}{2} t \|\nabla v(t)\|^2_{L^2}+\int_0^t\|\sqrt{s }v_{t}\|_{L^2(\Omega)}^2\ds \\
\lesssim  &\,\begin{multlined}[t]\frac{b}{2}\int_0^t \|\nabla v\|^2_{L^2} \ds+\int_0^t\int_\Omega s v_t f \dx\ds+ \int_0^t\|\sqrt{s}\nabla v\|_{L^2}^2\ds\\+\int_0^t\int_\Omega \left|\nabla \psi\cdot \nabla v\right|\dx\ds
+\int_\Omega \left|t\nabla \psi(t)\cdot\nabla v(t)\right|\dx .    \end{multlined}
\end{aligned}
\end{equation}
We can then estimate
\begin{equation}\label{f_1_Term_E_1_1}
\begin{aligned}   
\left|\int_0^t \intO  s f  v_t \dxs \right| 
\leq&\,\varepsilon\int_0^t \|\sqrt{s}v_{t}(s)\|_{L^2}^2\ds+C(\varepsilon)\int_0^t \|{\sqrt{s}} f(s)\|_{L^2}^2\ds.
\end{aligned}
\end{equation}
We also have 
\begin{equation}\label{psi_term_1}
\begin{aligned}
\int_0^t\int_\Omega \left| \nabla \psi\cdot \nabla v\right|\dx\ds\lesssim&\, \int_0^t \|\nabla \psi\|_{L^2}^2\ds+\int_0^t\|\nabla v\|_{L^2}^2\ds.
\end{aligned}
\end{equation}
Furthermore, we can use the derived bounds  \eqref{E_1_Theta_1} on $\nabla \psi$ and $\nabla v$ to find
\begin{equation}\label{psi_term_2}
\begin{aligned}
\int_\Omega \left|t\nabla \psi(t)\cdot\nabla v(t)\right|\dx \lesssim&\,  T (\|\nabla \psi(t)\|_{L^2}^2+\|\nabla v(t)\|_{L^2}^2)\\
\lesssim&\,  T (  \Vert  \psi_0 \Vert^2_{H^1}+\Vert  \psi _1\Vert^2_{H^1}+\nLtwoLtwo{ f }^2).
\end{aligned}
\end{equation}
The first term on the right of \eqref{f_1_Term_E_1_1} will be absorbed by the left-hand side of \eqref{E_1_Int} as long as $\varepsilon$ is small enough. 
We thus infer from \eqref{E_1_Int} by using estimates \eqref{f_1_Term_E_1_1}--\eqref{psi_term_2} that
\begin{equation}\label{E_1_Main}
\begin{aligned}
&t \|\nabla v(t)\|^2_{L^2}+\int_0^t\|\sqrt{s }v_{t}(s)\|_{L^2}^2\ds \\
\lesssim&\, \begin{multlined}[t] \int_0^T \|{\sqrt{s}}f(s)\|_{L^2}^2\ds + \int_0^t\|\nabla v(s)\|^2_{L^2}\ds
+ \int_0^t\|\sqrt{s}\nabla v\|_{L^2}^2\ds\\+T (\|\nabla \psi(t)\|_{L^2}^2+\|\nabla v(t)\|_{L^2}^2). \end{multlined}
\end{aligned}    
\end{equation}
An application of Gronwall's inequality yields \eqref{L_p_q_Estimates_Main}, where the hidden constant has the form $C(1+T)e^{CT}$.\\     
\indent It remains to prove estimate \eqref{Ineq_weight_nabla_v}. To this end, we take the time derivative of the heat equation and multiply it by $\sqrt{t}$:  
\begin{equation}\label{time_D_heat}
\partial_t(\sqrt{t} v_t)-\frac{1}{2\sqrt{t}} v_t -b\Delta \sqrt{t}v_t= \sqrt{t}f_t+c^2\Delta  \sqrt{t} v.
\end{equation}   
Multiplying \eqref{time_D_heat} by $\sqrt{t}v_t$ and integrating over $\Omega$ (keeping in mind that $v_t|_{\partial\Omega}=0$) then yields
\begin{equation}
\begin{aligned}
&\frac{1}{2}\ddt \Vert \sqrt{t}v_t\Vert_{L^2}^2 + b\Vert \sqrt{t}\nabla v_t\Vert_{L^2}^2\\
\lesssim&\, \Vert v_t\Vert_{L^2}^2+\varepsilon \Vert \sqrt{t}\nabla v_t\Vert_{L^2}^2+\Vert \sqrt{t}f_t\Vert_{H^{-1}}^2+ \sqrt{T}\|\nabla v\|^2_{L^2}, 
\end{aligned} 
\end{equation}
where we have used the estimate $\langle \sqrt{t}v, \sqrt{t}f\rangle_{H^{-1}, H^1}\leq \|\sqrt{t}\nabla v\|_{L^2} \|\sqrt{t}f\|_{H^{-1}}$.\\
\indent For small enough $\varepsilon>0$, by integrating over $t \in (0,T)$ and using \eqref{E_1_Theta_1} to bound $\Vert v_t\Vert_{L^2(L^2)}^2$ and $\nLtwoLtwo{\nabla v}$, we obtain \eqref{Ineq_weight_nabla_v}, thus completing the proof.
\end{proof}  
 Our aim now is to show that we can gain one spatial derivative in terms of regularity of $\psi_t$ with respect to the initial condition $\psi_1$, provided we pay the price of a time weight. To this end, we will establish sufficient conditions under which the solution of \eqref{v_Equation_0} satisfies
\begin{equation}\label{Regularity_H^2_Heat}
\sqrt{t} \D v \in L^\infty(0,T; L^2(\Omega)).
\end{equation} 
The corresponding bound on $\nLinfLtwo{\sqrt{t}\Delta v}$ will be crucial in the later analysis of the nonlinear problem.
\begin{proposition}\label{Prop:Lin_dataH2}
	Given  a final time $T>0$, let the initial conditions be
	\[
	(\psi_0, \psi_1) \in \Honetwo \times \Honezero,\] and the source term $	f \in L^2(0,T; L^2(\Omega))$. Then the following bound holds for the solution of \eqref{v_Equation_0}:
		\begin{equation}\label{Theta_Maximal_Regularity}
			\begin{aligned}
				 {\|\psi\|_{L^\infty(H^{2})}}+{ \|v\|_{L^2(H^{2})}} 
				\lesssim_T\, \Vert  \psi_0 \Vert_{H^{2}}+\Vert  \psi _1\Vert_{H^1}+\nLtwoLtwo{ f }.
			\end{aligned}
		\end{equation}
		If additionally $\sqrt{t} f \in L^{\infty}(0,T; L^{2}(\Omega))$, $\sqrt{t} f_t \in L^2(0,T; H^{-1}(\Omega))$
		for all $t \in (0,T)$, then
		\begin{equation}\label{H^2_weighted_Estimate}
			\begin{aligned}        
				\Vert \sqrt{t} \Delta v\Vert_{L^{\infty}(L^{2})}
				\lesssim_T &\, \begin{multlined}[t]  \|\psi_0\|_{H^2}+\Vert \psi_1 \Vert_{H^1}+\Vert f \Vert_{L^2(L^2)}
				+ \|\sqrt{t}f\|_{L^\infty(L^2)}\\
				+\Vert \sqrt{t}f_t\Vert_{L^2(H^{-1})}.
				\end{multlined} 
			\end{aligned}    
		\end{equation}
\end{proposition} 
 \begin{proof}
We conduct the proof by bootstrapping the regularity obtained in Proposition~\ref{Prop:Lin_dataH1}. To estimate $\Delta  v $, we write the nonlocal heat equation in  \eqref{v_Equation_0} in the form  
\begin{equation}\label{Laplacian_Eq}
-\Delta v-\frac{c^2}{b}\Delta \psi=- \frac{1}{b}v_t +\frac{1}{b}f.
\end{equation}
We then multiply it by $-\Delta v$ and use
$v=\psi_t$ to arrive at     
\begin{equation}\label{Delta_psi_Esti}
\|\Delta v\|_{L^2}^2+\frac{c^2}{2b}\ddt \|\Delta \psi\|_{L^2}^2=-\frac{1}{b}\int_\Omega v_t \Delta v\dx+\frac{1}{b}\int_\Omega f \Delta v\dx.       
\end{equation}
Young's inequality with $\varepsilon>0$ small enough yields, after integration in time,
\begin{equation}   
\nLinfLtwo{\Delta \psi}+\nLtwoLtwo{\Delta v}\lesssim \|\Delta \psi_0\|_{L^2}+\nLtwoLtwo{v_t}+\nLtwoLtwo{f}.
\end{equation}
Taking into account the estimate of $\nLtwoLtwo{v_t}$ in \eqref{E_1_Theta_1}, we obtain \eqref{Theta_Maximal_Regularity}. \\
\indent To prove estimate \eqref{H^2_weighted_Estimate},  we multiply \eqref{Laplacian_Eq} by $\sqrt{t}$:
\begin{equation}
-  \sqrt{t}\Delta v=\, -\frac{\sqrt{t}}{b}v_{t}+\frac{\sqrt{t}}{b}f+ \frac{c^2}{b} \sqrt{t}\Delta \psi.
\end{equation}
From here we immediately have
\begin{equation}\label{Weighted_Estimate_H_2}
\begin{aligned}
	\|\Delta \sqrt{t}v \|_{L^{\infty}(L^{2})} \lesssim&\,  \|\sqrt{t}f\|_{L^{\infty}(L^{2})}+\|\sqrt{t}v_t\|_{L^{\infty}(L^2)}  + \sqrt{T}\|\Delta \psi\|_{L^{\infty}(L^{{2}})}. 
\end{aligned} 	
\end{equation}  
Combining this bound with \eqref{Ineq_weight_nabla_v} and \eqref{Theta_Maximal_Regularity}  to estimate the last two terms on the right yields \eqref{H^2_weighted_Estimate}. 
\end{proof} 
We observe from the last proof that the assumption $\psi_0 \in H^2(\Omega)$ in the statement of Proposition~\ref{Prop:Lin_dataH2} above is due to the having the nonlocal term $-c^2\Delta \psi$ in the heat equation.  A bound on $\|\Delta \psi\|_{L^\infty(L^2)}$ will also be needed to estimate $f$ further using \eqref{f_form_1} and, in turn, tackle the nonlinear problem.\\
\indent Motivated by the previous analysis, let us introduce the time-weighted space $\Xtv \subset \Xv$ to which $v=\psi_t$ belongs:
\begin{equation} \label{space_Xv}
	\begin{aligned}
		\Xtv = \{ \, v \in \Xv: 
		&\, \nLinfLtwo{\sqrt{t}v_t }+ \nLtwoLtwo{ \sqrt{t} \nabla v_t}+ \nLinfLtwo{ \sqrt{t} \D v}<\infty \}
	\end{aligned}
\end{equation}
with the weight-independent contribution
\begin{equation} \label{space_Xv}
	\begin{aligned}
		\Xv = \{ v \in L^\infty(0,T; \Honezero) \cap L^2(0,T; \Honetwo): \, v_t \in L^2(0,T; \Ltwo) \}.
	\end{aligned}
\end{equation}
The corresponding norm is denoted by $\|\cdot \|_{\Xtv}$.  According to Propositions~\ref{Prop:Lin_dataH1} and \ref{Prop:Lin_dataH2}, we then have
\begin{equation} \label{final_est_v}
	\begin{aligned}
		\|v\|_{\Xtv} \lesssim_T  \begin{multlined}[t] \|\psi_0\|_{H^2}+\Vert \psi_1 \Vert_{H^1}+\Vert f \Vert_{L^2(L^2)}\\
			+\nLinfLtwo{\sqrt{t}f}+\Vert \sqrt{t}f_t\Vert_{L^2(H^{-1})}. \end{multlined}
	\end{aligned}
\end{equation}
\subsection{Estimates for the linearized Blackstock equation}
Our next aim is to derive time-weighted bounds for \eqref{v_Equation} by relying on the obtained estimates for the nonlocal heat equation but now using the form of $f$ given in \eqref{f_form_1}. The solution space for the acoustic velocity potential will be  $\Xt \subset \Xpsi$, defined by
\begin{equation} \label{space_Xt}
	\begin{aligned}
	 \Xt 
		=\, \{\psi \in \Xpsi:&\, \begin{multlined}[t]  \nLinfLtwo{\sqrt{t}\psi_{tt}} +\nLtwoLtwo{\sqrt{t} \nabla \psi_{tt}}+\nLinfLtwo{\sqrt{t} \D \psi_t } < \infty\} \end{multlined}
	\end{aligned}
\end{equation} 
with the weight-independent contribution
\begin{equation} \label{space_Xpsi}
	\begin{aligned}
	\Xpsi
		=\, \{\psi:&\, \psi \in L^\infty(0,T; \Honetwo), \\
		&\, \psi_t \in L^\infty(0,T; \Honezero) \cap L^2(0,T; \Honetwo), \\
		&\, \psi_{tt} \in L^2(0,T; \Ltwo) \}.
	\end{aligned}
\end{equation}
We next prove well-posedness of the linearized Blackstock problem in $\Xt$. 
\begin{proposition}\label{Proposition_apriori_Estimates}
 Let $T>0$ and let assumption \eqref{assumptions_coefficients} on the medium coefficients hold. Assume that \[(\psi_0, \psi_1)\in \Honetwo\times H^1_{0}(\Omega)\]
 and let 
 	\begin{equation}
 		\begin{aligned}
	\tilde{f} \in \{ \tilde{f}\in L^2(0,T; L^2(\Omega)): \nLinfLtwo{ \sqrt{t} \tilde{f} }+\nLtwoHneg{ \sqrt{t} \tilde{f}_t}<\infty\}.
 \end{aligned}	
\end{equation} 
  Furthermore, assume that there exists $R>0$, such that
\begin{equation}
	\|\alpha\|_{\Xtv} \leq R.
\end{equation}
Then there exists $m=m(R, T)>0$, 
 such that if the coefficient $\alpha$ is sufficiently small in the sense of
\begin{equation} \label{cond_smallness_m}
	\begin{aligned}
|k| (\|\alpha\|_{L^\infty(L^2)}+\|\sqrt{t}\alpha_t\|_{L^2(L^2)}) + |\sigma|(\|\nabla \alpha\|_{L^2(L^2)}+\|\sqrt{t}\nabla \alpha_t\|_{L^2(L^2)}) \leq m, 
	\end{aligned}
\end{equation}
then there is a unique $\psi \in \Xt$ 
which solves
	\begin{equation} \label{linear_ibvp}
\left \{ \begin{aligned}
& \psi_{tt}-c^2(1-2k \alpha(x,t))\Delta \psi - b \Delta \psi_t+ 2 \sigma \nabla \psi \cdot \nabla \alpha(x,t)=  \tilde{f} \ \text{ in } \, \Omega \times(0,T), \\    
&(\psi, \psi_t)\vert_{t=0}=(\psi_0, \psi_1),\\  
&\psi_{\vert \partial \Omega}= 0.    
\end{aligned} \right. 
\end{equation}
This solution satisfies the following bound:
\begin{equation} \label{final_est_linear}
\begin{aligned}
\|\psi\|_{\Xt}
 \lesssim_T&\, \begin{multlined}[t] \|\psi_0\|_{H^2}+ \|\psi_1\|_{H^1} 
	+\Vert \tilde{f} \Vert_{L^2(L^2)} +\Vert\sqrt{t} \tilde{f} \Vert_{L^\infty(L^2)}
+\Vert \sqrt{t}\tilde{f}_t\Vert_{L^2(H^{-1})} .\end{multlined}
\end{aligned}
\end{equation}
\end{proposition}
\begin{proof}
By combining estimates \eqref{Theta_Maximal_Regularity} and \eqref{final_est_v}, we obtain 
\begin{equation}\label{Xt_est_psit}
\begin{aligned}
\|\psi\|_{\Xt} \lesssim_T&\, \nLinfLtwo{\Delta \psi}+ \|v\|_{\Xtv}\\ \lesssim_T& \, \begin{multlined}[t]   \|\psi_0\|_{H^2}+\Vert  \psi _1\Vert_{H^1(\Omega)}+\Vert f \Vert_{L^2(L^2)}+\nLinfLtwo{\sqrt{t}f} 
 +\Vert \sqrt{t}f_t\Vert_{L^2(H^{-1})}. \end{multlined}
\end{aligned}
\end{equation}
Thus the proof boils down to estimating the $f$ terms on the right-hand side above. Recall that
\begin{equation} \label{def_f}
f= -2k c^2\alpha(x,t)\Delta \psi-2 \sigma \nabla \psi \cdot \nabla \alpha(x,t)+\tilde{f}.    
\end{equation}  
H\"older's inequality and interpolation estimates \eqref{intp_ineq_LtwoLinf}  allow us to conclude that
\begin{equation}\label{f_L_2}
\begin{aligned}  
\nLtwoLtwo{f}\lesssim&\,{|k|}\nLtwoLinf{\alpha}\nLinfLtwo{\Delta \psi}+|\sigma|\nLinfLfour{\nabla \psi}\nLtwoLfour{\nabla \alpha}+\nLtwoLtwo{\tilde{f}}\\
\lesssim&\,\begin{multlined}[t]|k| \nLtwoLtwo{ \D \alpha}^{d/4}\nLtwoLtwo{\alpha}^{1-d/4} \nLinfLtwo{\Delta \psi}\\+ |\sigma|\Vert \nabla \alpha\Vert_{L^2(H^1)}^{d/4}\nLtwoLtwo{ \nabla \alpha}^{1-d/4}\nLinfLfour{\nabla \psi}+\nLtwoLtwo{\tilde{f}}. 
\end{multlined}
 \end{aligned}
\end{equation}
Employing additionally Poincar\'e's inequality and the embeddings $H^2(\Omega) \hookrightarrow H^1(\Omega) \hookrightarrow  L^4(\Omega)$ together with elliptic regularity yields  
\begin{equation}\label{f_L_2}
\begin{aligned}  
\nLtwoLtwo{f}
\lesssim&\, m^{1-d/4} \nLtwoLtwo{ \D \alpha}^{d/4} \nLinfLtwo{\Delta \psi}+\nLtwoLtwo{\tilde{f}}. 
\end{aligned}
\end{equation}
We next estimate $\nLinfLtwo{\sqrt{t}f}$ in \eqref{Xt_est_psit}. H\"older's and Agmon's inequalities imply
\begin{equation} \label{interim_est_sqrtt_f}
\begin{aligned}
&\Vert\sqrt{t}f(t) \Vert_{L^{2}}\\
\lesssim&\, |k| \|\sqrt{t}\alpha(t)\|_{L^\infty} \Vert \Delta \psi(t) \Vert_{L^2}+|\sigma|\Vert \sqrt{t} \nabla \alpha(t)\Vert_{L^{4}} \|\nabla \psi(t)\|_{L^4}+\Vert\sqrt{t}\tilde{f}(t) \Vert_{L^2}\\    
\lesssim_T&\,  \begin{multlined}[t]|k|\|\sqrt{t}\alpha(t)\|_{H^2}^{d/4}\|\alpha(t)\|_{L^2}^{1-d/4} \Vert \Delta \psi(t) \Vert_{L^2}
+|\sigma|\Vert \sqrt{t} \nabla \alpha(t)\Vert_{L^{4}} \|\nabla \psi(t)\|_{L^4}\\+\Vert\sqrt{t}\tilde{f}(t) \Vert_{L^2}. \end{multlined}
\end{aligned}  
\end{equation}
Above in the last line we have used
\[
\|\sqrt{t} \alpha(t)\|_{L^2} \leq \sqrt{T} \|\alpha(t)\|_{L^2}.
\]
Using  Lemma~\ref{Lemma_Interpolation} with $q=4$ together with H\"older's inequality in time, we obtain 
\begin{equation}\label{L_4_d_3}
\begin{aligned}
|\sigma|\Vert \sqrt{t} \nabla \alpha\Vert_{L^\infty(L^{4})} \lesssim |\sigma|\Vert \sqrt{t} \nabla \alpha\Vert^{
1-d/4}_{L^\infty(L^{2})} \Vert \sqrt{t} \nabla \alpha\Vert^{ d
/4}_{L^\infty(H^{1})} \lesssim  m^{
1-d/4} \Vert \sqrt{t} \Delta \alpha\Vert^{ 
d/4}_{L^\infty(L^2)}. 
\end{aligned}
\end{equation}
These estimates employed in \eqref{interim_est_sqrtt_f} yield
\begin{equation}
\begin{aligned}
\nLinfLtwo{\sqrt{t}f}
\lesssim_T\,\begin{multlined}[t] m^{1-d/4} \|\sqrt{t}\D \alpha\|_{L^\infty(L^2)}^{{d/4}} \nLinfLtwo{ \Delta \psi}+\nLinfLtwo{\sqrt{t}\tilde{f}}.\end{multlined}   
\end{aligned}
\end{equation}
Next we estimate $\|\sqrt{t}f_t\|_{L^2(H^{-1})}$. To this end, we rely on the following inequality:
\begin{equation}\label{Embedding_H_1}
	\begin{aligned}   
	\|ab\|_{H^{-1}} \lesssim&\,  \|ab\|_{L^{6/5}} \lesssim \|a\|_{L^2}\|b\|_{L^3}   \quad a \in L^2(\Omega), b \in L^3(\Omega).
	\end{aligned}
	\end{equation}
 Since 
\[      
	f_t = -2k c^2 \alpha_t\Delta \psi-2 \sigma \nabla \psi_t \cdot \nabla \alpha(x,t)-2kc^2\alpha(x,t)\Delta \psi_t-2 \sigma \nabla \psi\cdot \nabla \alpha_t(x,t)+\tilde{f}_t   
	\]
the use of estimate \eqref{Embedding_H_1} together with H\"older's inequality implies
	\begin{equation} 
	\begin{aligned}    
	&\Vert \sqrt{t} f_t\Vert_{L^2(H^{-1})}\\
	 \lesssim& \, \begin{multlined}[t]|k| \|\sqrt{t}\alpha_t\|_{L^2(L^3)}\|\Delta \psi\|_{L^\infty(L^2)}+ |\sigma|\nLinfLtwo{\sqrt{t}\nabla \psi_t}\|\nabla \alpha\|_{L^2(L^{3})}\\+|k|\|\alpha\|_{L^2(L^{3})}\nLinfLtwo{\sqrt{t}\Delta \psi_t}+|\sigma|\|\nabla \psi\|_{L^\infty(L^3)}\|\sqrt{t}\nabla \alpha_t\|_{L^2(L^2)}+	\Vert \sqrt{t} \tilde{f}_t\Vert_{L^2(H^{-1})}.
	 \end{multlined}        
	\end{aligned}   
	\end{equation} 
	We have by using Lemma \ref{Lemma_Interpolation} together with the elliptic regularity  
\begin{equation} \label{L_3_gradient_Alpha}
\begin{aligned}
{|\sigma|}\|\nabla \alpha \|_{L^2(L^3)} \lesssim{|\sigma|}\|\nabla \alpha\|_{L^2(L^2)}^{{1-\frac{d}{6}}} \|\Delta \alpha\|_{L^2(L^2)}^{{\frac{d}{6}}
} \lesssim  m^{{1-\frac{d}{6}
}} \|\Delta \alpha\|_{L^2(L^2)}^{{\frac{d}{6}}}.
\end{aligned}
\end{equation}
Similarly,
\begin{equation} \label{L_3_Alpha}
	\begin{aligned}
{|k|}	\| \alpha \|_{L^2(L^3)} \lesssim 	{|k|}\| \alpha\|_{L^2(L^2)}^{1-\frac{d}{6}} \|\nabla \alpha\|_{L^2(L^2)}^{\frac{d}{6}
		} \lesssim  m^{1-\frac{d}{6}} \|\nabla \alpha\|_{L^2(L^2)}^{\frac{d}{6}
		}
	\end{aligned}
\end{equation}
and
\begin{equation} \label{L_3_Alphat}
	\begin{aligned}
		{|k|}\|\sqrt{t}\alpha_t \|_{L^2(L^3)} \lesssim {|k|}\|\sqrt{t}\alpha_t\|_{L^2(L^2)}^{{1-\frac{d}{6}}} \|\sqrt{t}\alpha_t\|_{L^2(H^1)}^{{\frac{d}{6}}
		}\lesssim  m^{{1-\frac{d}{6}}} \|\sqrt{t} \nabla \alpha_t\|_{L^2(L^2)}^{{\frac{d}{6}}}.
	\end{aligned}
\end{equation} 
Thus  we have by using \eqref{L_3_gradient_Alpha}--\eqref{L_3_Alphat} and elliptic regularity,
\begin{equation}\label{f_t_Term}
	\begin{aligned}
	\Vert \sqrt{t} f_t\Vert_{L^2(H^{-1})}
	 \lesssim& \,	\begin{multlined}[t]  m^{1-\frac{d}{6}} \|\sqrt{t}\nabla \alpha_t\|_{L^2(L^2)}^{\frac{d}{6}} \nLinfLtwo{ \Delta \psi}
	+{ m^{{1-\frac{d}{6}}} } \|\Delta \alpha\|_{L^2(L^2)}^{{\frac{d}{6}}}\|v\|_{\Xtv}\\+	\Vert \sqrt{t} \tilde{f}_t\Vert_{L^2(H^{-1})}.  \end{multlined}
	\end{aligned}
\end{equation} 
Inserting all the derived bounds on $f$ terms into \eqref{Xt_est_psit} yields
\begin{equation} \label{interim_est_psi_t}
\begin{aligned}
&\nLinfLtwo{\D \psi}+\|v\|_{\Xtv}\\ \lesssim_T& \, \begin{multlined}[t]  {\|\psi_0\|_{H^2}+ \Vert  \psi _1\Vert_{H^1}}+\Lambda[\alpha,  m]( \nLinfLtwo{ \Delta \psi} +\|v\|_{\Xtv}) 
\\ +\nLinfLtwo{\sqrt{t} \tilde{f} }
	+\Vert \tilde{f} \Vert_{L^2(L^2)}
	+\Vert \sqrt{t} \tilde{f}\Vert_{L^2(H^{-1})}+\Vert \sqrt{t}\tilde{f}_t\Vert_{L^2(H^{-1})}\end{multlined}
\end{aligned}
\end{equation}
with 
\begin{equation}
\begin{aligned}
\Lambda[\alpha,  m]=&\,\begin{multlined}[t] \max \{m^{1-d/4}, m^{d/4}, m^{1-d/6}\} \|\alpha\|_{\Xtv}.
\end{multlined}    
\end{aligned}
\end{equation}	
Thus, from \eqref{interim_est_psi_t} for sufficiently small $m= m(\|\alpha\|_{\Xtv}, T)>0$, we obtain
\begin{equation}
	\begin{aligned}
\nLinfLtwo{\Delta \psi}+\|\psi_t\|_{\Xtv} \lesssim_T& \, \begin{multlined}[t]   \|\psi_0\|_{H^2}+\|\psi_1\|_{H^1}\\ +\nLinfLtwo{\sqrt{t} \tilde{f} }
+\Vert \tilde{f} \Vert_{L^2(L^2)}
+\Vert \sqrt{t}\tilde{f}_t\Vert_{L^2(H^{-1})}, \end{multlined}
\end{aligned}
\end{equation}
from which  \eqref{final_est_linear} follows. We note that if $\sigma=0$, a smallness assumption on $\|\nabla \alpha\|_{L^2(L^2)}+\|\sqrt{t} \nabla \alpha_t\|_{L^2(L^2)}$  is not needed. Of course, if both $k=\sigma=0$, the smallness condition in the statement is trivially satisfied.
\end{proof}
\section{A fixed-point argument} \label{Sec:NonLinProblem}
To relate the previous analysis to the nonlinear problem, we employ the Banach fixed-point theorem under the assumption of small enough data.
\begin{theorem}[Local solvability of the Blackstock equation] \label{Thm:LocalWellP} Let $T>0$ and
	\begin{equation}\label{Initial_Regularity}
	(\psi_0, \psi_1) \in \Honetwo \times \Hone.
	\end{equation}
Let the medium coefficients satisfy \eqref{assumptions_coefficients}. There exists $\delta=\delta(T)>0$,
 such that if data is sufficiently small in the sense of
\begin{equation}\label{Smallness_Assumption}
\|\psi_0\|_{H^2}+ \|\psi_1\|_{H^1} \leq \delta,
\end{equation}
then there is a unique $\psi \in \Xt$ 
 which solves
\begin{equation} \label{nonlinear_ibvp}
	\left \{ \begin{aligned}
		& \psi_{tt}-c^2(1-2k \psi_t)\Delta \psi - b \Delta \psi_t+ 2 \sigma \nabla \psi \cdot \nabla \psi_t= 0 \quad \text{in} \ \Omega \times (0,T), \\
		&(\psi, \psi_t)=(\psi_0, \psi_1),\\  
		&\psi_{\vert \partial \Omega}= 0,
	\end{aligned} \right. 
\end{equation}
with $\Xt \subset \Xpsi$ defined in \eqref{space_Xt}. The solution depends continuously on the initial data with respect to the $\|\cdot \|_{\Xt}$ norm. 
\end{theorem}
Before moving onto the proof, we briefly discuss the statement made above.
\begin{itemize}[leftmargin=8mm]
\item Theorem~\ref{Thm:LocalWellP} guarantees solvability under weaker regularity assumptions on initial conditions than those available in the literature~\cite{fritz2018well, kawashima1992global, tani2017mathematical}, where the initial data is assumed to have at least the regularity given in \eqref{H_3_Intial}. 
\item Although the final time $T$ is fixed,  there are no restrictions on its  size. 
\item The presence of the time weights yields the additional higher regularity of the solution so that $\psi\in \Xt$ and not only $\psi \in \Xpsi$. Without the  developed time-weighted framework, such a regularity cannot be shown for initial data satisfying \eqref{Initial_Regularity}.  
\end{itemize}
\begin{proof}
As announced, we set up a fixed-point mapping
	\begin{equation}
	\begin{aligned}
	\mathcal{T}: \mathcal{B} \ni \psi^* \mapsto \psi,
	\end{aligned}
	\end{equation}
	where    
	\begin{equation} \label{def_ball}
	\begin{aligned}
	\mathcal{B} =\left \{\psi^* \in \Xt \right.:& \, \|\psi^*\|_{\Xt} \leq R, \ (\psi^*, \psi_t^*)=(\psi_0, \psi_1), \\
	 &  \, \begin{multlined}[t]
	|k| (\|\psi_t^*\|_{L^\infty(L^2)}+\|\sqrt{t}\psi_{tt}^*\|_{L^2(L^2)})\\ + |\sigma|(\|\nabla \psi_t^*\|_{L^2(L^2)}+\|\sqrt{t}\nabla \psi_{tt}^*\|_{L^2(L^2)}) \leq m  \left\}  \vphantom{\Xt}  \right. \end{multlined}
	\end{aligned}
	\end{equation}
	and $\psi$ solves the linear problem  \eqref{linear_ibvp} with $\tilde{f}=0$ and the variable coefficient $\alpha = \psi^*_t$:
		\begin{equation} \label{linear_ibvp_fp}
		\left \{ \begin{aligned}
			& \psi_{tt}-c^2(1-2k \psi^*_t)\Delta \psi - b \Delta \psi_t+ 2 \sigma \nabla \psi \cdot \nabla \psi^*_t=  0 \ \text{ in } \, \Omega \times(0,T), \\    
			&(\psi, \psi_t)=(\psi_0, \psi_1),\\  
			&\psi_{\vert \partial \Omega}= 0.    
		\end{aligned} \right. 
	\end{equation}
	It is suffices  to find a (unique) fixed point of the mapping $\mathcal{T}(\psi^\ast)=\psi$. 
	We choose $m>0$ in \eqref{def_ball} according to Proposition~\ref{Proposition_apriori_Estimates} which guarantees that the mapping is well-defined (and $\mathcal{B}$ non-empty).\\
	\indent  Take $\psi^\ast\in \mathcal{B}$.  To prove the self-mapping property, we rely on Proposition~\ref{Proposition_apriori_Estimates}. We choose $R>0$  so that
	\begin{equation}\label{psi_R_bound}
	R \geq  C_{\textup{lin}}(T)(\|\psi_0\|_{H^2}+\|\psi_1\|_{H^1} )\geq \|\psi\|_{\Xt},
	\end{equation}
where $C_{\textup{lin}}(T)$ is the hidden constant in \eqref{final_est_linear}. To prove that $\psi$ satisfies the $m$ bound within \eqref{def_ball}, we note that
\begin{equation} \label{m_bound_psi}
	\begin{aligned}
\begin{multlined}[t]	|k| (\|\psi_t\|_{L^\infty(L^2)}+\|\sqrt{t}\psi_{tt}\|_{L^2(L^2)})\\ + |\sigma|(\|\nabla \psi_t\|_{L^2(L^2)}+\|\sqrt{t}\nabla \psi_{tt}\|_{L^2(L^2)}) \lesssim \|\psi\|_{\Xt}. 
	\end{multlined}
	\end{aligned}
\end{equation}
Thus, energy bound \eqref{final_est_linear} for the linearized problem guarantees that
\begin{equation}\label{m_bound_psi}
\begin{aligned}
	\|\psi\|_{\Xt} \leq C_{\textup{lin}}(T)(\|\psi_0\|_{H^2}+\|\psi_1\|_{H^1} )\lesssim C_{\textup{lin}}(T) \delta  \leq m
\end{aligned}
\end{equation}
by reducing the size of data $\delta$. Hence, \eqref{psi_R_bound} together with \eqref{m_bound_psi} shows that $\psi\in \mathcal{B}$. \\
	\indent In the second part of the proof, we prove strict contractivity. Take $\varphi^*$, $\phi^* \in 
	\mathcal{B}$ and let $\mathcal{T}(\varphi^*)=\varphi$, $\mathcal{T}(\phi^*)=\phi$. We also introduce the differences
	\begin{equation}
		\bar{\psi}=\varphi-\phi, \qquad \bar{\psi}^*=\varphi^*-\phi^*.
	\end{equation}   
	Then $\bar{\psi} \in \mathcal{B}$ solves
	\begin{equation}
		\begin{aligned}
			\bar{\psi}_{tt}-c^2(1-2k \varphi^*_t)\Delta 	\bar{\psi}-b \Delta 	\bar{\psi}_t + 2  \sigma \nabla 	\bar{\psi} \cdot \nabla \varphi_t^*=- 2k c^2 \bar{\psi}^*_t \Delta \phi - 2\sigma \nabla \phi \cdot \nabla \bar{\psi}^*_t
		\end{aligned}
	\end{equation}
with homogeneous boundary and initial conditions.  We can thus employ estimate \eqref{final_est_linear} with zero initial data, that is
\begin{equation} 
	\begin{aligned}
		\|\bar{\psi}\|_{\Xt}	\lesssim_T &\, \begin{multlined}[t] \nLinfLr{\sqrt{t} \tilde{f}}
				+\Vert \tilde{f} \Vert_{L^2(L^2)}     
				+\Vert \sqrt{t} \tilde{f}\Vert_{L^2(H^{-1})}+\Vert \sqrt{t}\tilde{f}_t\Vert_{L^2(H^{-1})}, \end{multlined}
	\end{aligned}
\end{equation}
where 
\begin{equation}
	\begin{aligned}
		\tilde{f} = - 2k c^2 \bar{\psi}^*_t \Delta \phi - 2\sigma \nabla \phi \cdot \nabla \bar{\psi}^*_t.
	\end{aligned}
\end{equation}
It remains to estimate the $\tilde{f}$ terms, which we can do similarly to the estimates of $f$ terms in \eqref{Xt_est_psit} in the proof of Proposition~\ref{Proposition_apriori_Estimates}. We have
\begin{equation}
	\begin{aligned}
	\Vert \tilde{f} \Vert_{L^2(L^2)} \lesssim&\,  \|\psi_t^*\|_{L^2(L^\infty)} \|\Delta \phi\|_{L^\infty(L^2)}+\|\nabla \phi\|_{L^\infty(L^4)}\|\nabla \bar{\psi}_t^*\|_{L^2(L^4)} \\
	\lesssim&\, \|\Delta \phi\|_{L^\infty(L^2)}\|\bar{\psi}^*\|_{\Xt}  \\
	\lesssim&\, R \|\bar{\psi}^*\|_{\Xt}.
	\end{aligned}    
\end{equation}
Next, 
\begin{equation}
	\begin{aligned}
	\nLinfLtwo{\sqrt{t} \tilde{f}} \lesssim&\, \|\sqrt{t}\psi_t^*\|_{L^\infty(L^2)} \|\Delta \phi\|_{L^\infty(L^2)}+\|\nabla \phi\|_{L^\infty(L^{4})}\|\sqrt{t}\nabla \bar{\psi}_t^*\|_{L^\infty(L^2)}\\
		 \lesssim&\, R \|\bar{\psi}^*\|_{\Xt}.
	\end{aligned}
\end{equation}
Additionally,
\begin{equation}
	\begin{aligned}
	\Vert \sqrt{t}\tilde{f}_t\Vert_{L^2(H^{-1})} =&\, \Vert \sqrt{t}(- 2k c^2 \bar{\psi}^*_{tt} \Delta \phi - 2k c^2 \bar{\psi}^*_t \Delta \phi_t- 2\sigma \nabla \phi_t \cdot \nabla \bar{\psi}^*_t- 2\sigma \nabla \phi \cdot \nabla \bar{\psi}^*_{tt})\Vert_{L^2(H^{-1})} \\
	\lesssim&\, \begin{multlined}[t]\|\sqrt{t} \bar{\psi}^*_{tt}\|_{L^2(L^3)}\nLinfLtwo{\Delta \phi}+\|\bar{\psi}_t^*\|_{L^2(L^{3})}\|\sqrt{t} \Delta \phi_t\|_{L^\infty(L^2)}\\
	+  \|\sqrt{t} \nabla \phi_t\|_{L^\infty(L^2)}\|\nabla \bar{\psi}_t^*\|_{L^2(L^3)}+\|\nabla \phi\|_{L^\infty(L^3)}\|\sqrt{t}\nabla \bar{\psi}_{tt}^*\|_{L^2(L^2)} \end{multlined} \\
\lesssim&\, R  \|\bar{\psi}^*\|_{\Xt}.
	\end{aligned}
\end{equation}
Therefore, we can guarantee strict contractivity of $\mathcal{T}$ with respect to the $\|\cdot \|_{\Xt}$ norm by reducing the radius $R$, which in turn requires sufficient smallness of $\delta$. By Banach's fixed-point theorem, we obtain a unique $\psi \in \mathcal{B}$, which solves \eqref{nonlinear_ibvp}.
	\end{proof}

\section{Global existence} \label{Sec:Global}
	To conclude, we discuss the global solvability of the nonlinear problem \eqref{Main_Problem}.  
 Our goal is to control the solution of \eqref{Main_Problem} uniformly as $t\rightarrow \infty$ in  a suitable energy norm.
 In addition, we accurately describe the asymptotic behavior of the solution of \eqref{Main_Problem} as $t\rightarrow \infty$. More precisely, we  show that the solutions decays exponentially fast in time.  To state the global result, we introduce energy $E(t)%
$ and the corresponding dissipation $D(t)$ at time $t \in (0,T)$ as
follows:
\begin{equation}
E(t)= \frac{1}{2}\|\psi_t(t)\|_{L^2}^2+\frac{c^2}{2}\|\nabla \psi(t)\|_{L^2}+\frac{c^2}{2b}\|\Delta \psi(t)\|_{L^2}^2+\|\nabla \psi_t(t)\|_{L^2}^2
\end{equation}
and 
\begin{equation}
D(t)=\int_0^t(\|\nabla \psi_t\|_{L^2}^2+ \|\Delta \psi_t\|_{L^2}^2+ \|\nabla \psi\|_{L^2}^2+\|\Delta \psi\|_{L^2}^2+\|\psi_{tt}\|_{L^2}^2)\ds.
\end{equation}
\begin{theorem}[Global solvability of the Blackstock equation]\label{Thm:GlobalWellP}
Assume that 
\begin{equation}
	(\psi_0, \psi_1) \in \Honetwo \times \Honezero.
	\end{equation}
	There exists $\epsilon_0>0$, such that if the data is sufficiently small so that
\begin{equation}
\|\psi_0\|_{H^2}+ \|\psi_1\|_{H^1} \leq \epsilon_0,
\end{equation}
then there is a unique global solution  $\psi$ of \eqref{Main_Problem}, such that 
\begin{equation}
\begin{aligned}
&\psi \in L^\infty(0,\infty; \Honetwo), 
\, \psi_t \in L^\infty(0,\infty; \Honezero) \cap L^2(0,\infty; \Honetwo), \\
&\psi_{tt} \in L^2(0,\infty; \Ltwo). 
\end{aligned}
\end{equation}
In addition, there exists a constant  $\zeta>0$, such that for all $t\geq 0$, we have  
\begin{equation}
E(t)\leq C E(0) e^{-\zeta t},
\end{equation}
where $C>0$ does not depend on time. 
\end{theorem}
\begin{proof}
The proof relies on the construction of suitable compensating functions $F_i=F_i(t)$ for $i=1,2,3$ that can capture  the  dissipation properties of problem \eqref{Main_Problem}. A Lyapunov function $L=L(t)$ can then be constructed as a linear combination of these functionals (with appropriate weights) and of the total energy $E=E(t)$.  As the function $L$ is equivalent to the energy, it allows recovering the optimal dissipation of the Blackstock equation. In addition, it satisfies a differential inequality that facilitates the exponential decay of the energy norm of the solution.   Below $C>0$ denotes a generic constant independent of time. Let
	\begin{equation}
E_1(t)=\frac{1}{2}\|\psi_t(t)\|_{L^2}^2+\frac{c^2}{2}\|\nabla \psi(t)\|_{L^2}. 
\end{equation}
Recall from \eqref{Testing_v_1} that  multiplying \eqref{B_eq} by $\psi_t$, integrating over $\Omega$, and using integration by parts yields
	    \begin{equation}
\ddt E_1(t)+b\|\nabla \psi_t\|_{L^2}^2=\int_\Omega f\psi_t \dx,
\end{equation}
where 
\begin{equation}\label{f_form}
f= -2kc^2 \psi_t \Delta \psi-2 \sigma \nabla \psi \cdot \nabla \psi_t. 
\end{equation}
Thus  by Young's and Poincar\'e's inequalities, we have  
\begin{equation}\label{E_1}
\ddt E_1(t)+\frac{b}{2}\|\nabla \psi_t\|_{L^2}^2\lesssim \| f\|_{L^2}^2. 
\end{equation}
Let
\begin{equation}
E_2(t)= \frac{c^2}{2b}\|\Delta \psi(t)\|_{L^2}^2. 
\end{equation}
We have from \eqref{Delta_psi_Esti}, 
\begin{equation}\label{E_2}
\ddt E_2(t)+\|\Delta \psi_t\|_{L^2}^2\leq C(\|\psi_{tt}\|_{L^2}^2+\| f\|_{L^2}^2).
\end{equation}
Next we introduce 
\begin{equation}\label{F_1}
F_1=\int_{\Omega}( \psi\psi_t+\frac12 b|\nabla \psi|^2)\dx.
\end{equation}
By testing \eqref{B_eq} by $\psi$, we immediately have
\begin{equation}
\ddt F_1(t)+c^2\|\nabla \psi\|_{L^2}^2=\|\psi_t\|_{L^2}^2+\int_{\Omega} f\psi\dx.
\end{equation}
Hence by Young's and Poincar\'e's inequalities we have  
\begin{equation}\label{F_1}
\ddt F_1(t)+\frac{c^2}{2}\|\nabla \psi\|_{L^2}^2\lesssim \|\psi_t\|_{L^2}^2+\|f\|_{L^2}^2 \lesssim \|\nabla \psi_t\|_{L^2}^2+\|f\|_{L^2}^2.
\end{equation}
We further introduce the functional
\begin{equation}\label{F_2}
F_2(t)= \int_\Omega \left(-\Delta \psi \psi_t+\frac{b}{2}|\Delta \psi|^2\right) \dx. 
\end{equation}
By testing \eqref{B_eq} by $-\Delta \psi$, we can see that  
\begin{equation}
\ddt F_2(t)+c^2\|\Delta \psi\|_{L^2}^2=\|\nabla \psi_t\|_{L^2}^2-\int_\Omega \Delta \psi f\dx,  
\end{equation}
which  yields 
\begin{equation} \label{F2_identity}
\ddt F_2(t)+\|\Delta \psi\|_{L^2}^2\lesssim \|\nabla \psi_t\|_{L^2}^2+\| f\|_{L^2}^2\dx.   
\end{equation}
To capture further dissipation terms, we also introduce 
\begin{equation}\label{F_3}
F_3= c^2\int_\Omega \nabla \psi\cdot \nabla \psi_{t}\dx+ \frac{b}{2}\int_\Omega |\nabla \psi_t |^2\dx.
\end{equation}
Then from \eqref{test_vt} we know that
\begin{equation} 
\begin{aligned}
\ddt F_3(t)+\int_\Omega \psi_{tt}^2\dx=&\,c^2\|\nabla \psi_t\|_{L^2}^2+\int_\Omega  \psi _{tt} f\dx
\end{aligned}
\end{equation}
and thus
\begin{equation} \label{test_psitt}
\begin{aligned}
		\ddt F_3(t)+\int_\Omega \psi_{tt}^2\dx \lesssim\,\|\nabla \psi_t\|_{L^2}^2+ \|f\|_{L^2}^2. 
\end{aligned}
\end{equation}
Let $\gamma_i$ for $i \in \{1,2,3\}$ be small positive constants. We define the Lyapunov functional
\begin{equation} \label{Lyapunov_f}
L(t)=E_1(t)+\gamma_1E_2(t)+\gamma_2 F_1(t)+\gamma_2 F_2(t)+\gamma_3 F_3(t),
\end{equation}
which we will show is equivalent to the energy $E$. We have by Poincar\'e's inequality  
\begin{equation}
\begin{aligned}
&\Big|L(t)-E_1(t)-\gamma_1E_2(t)-\gamma_3\frac{b}{2}\|\nabla \psi_t\|_{L^2}^2 \Big|\\
\leq&\, \gamma_2(|F_1(t)|+|F_2(t)|)+\gamma_3\left|c^2\int_\Omega \nabla \psi\cdot \nabla \psi_{t}\dx\right|\\
\leq&\,C\gamma_2(\|\psi_t\|_{L^2}^2+\|\nabla\psi\|_{L^2}^2+\|\Delta\psi\|_{L^2}^2)+\gamma_3\frac{b}{4}\| \nabla \psi_t\|_{L^2}^2+C\gamma_3\|\nabla \psi\|_{L^2}^2\\
\leq&\, C\gamma_2 (E_1(t)+E_2(t))+\gamma_3\frac{b}{4}\|\nabla  \psi_t\|_{L^2}^2.
 \end{aligned}
\end{equation}
Hence, this estimate yields 
\begin{equation}
(1-C\gamma_2-C\gamma_3)E_1(t)+(\gamma_1-C\gamma_2) E_2(t) +\gamma_3\frac{b}{4}\|\nabla \psi_t\|_{L^2}^2\leq L(t)\lesssim CE(t).
\end{equation}
 We fix $\gamma_2>0$ and $\gamma_3>0$ small enough so that
\[
\gamma_2+\gamma_3 <1/C
\]
and $\gamma_1$ large enough so that
\[
\gamma_1 >C\gamma_2.
\]
Then for all $t\geq 0$ we have the equivalence
\begin{equation}\label{Eq_L_E}
C_1E(t)\leq L(t)\leq C_2E(t)
\end{equation}
 for some $C_1$, $C_2>0$, independent of time. From \eqref{Lyapunov_f} and the derived bounds, we conclude that
\begin{equation}\label{Diffe_Ineq}
\begin{aligned}
&\ddt L(t)+ b\|\nabla \psi_t\|_{L^2}^2+ \gamma_1\|\Delta \psi_t\|_{L^2}^2+\gamma_2{\frac{c^2}{2}} \|\nabla \psi\|_{L^2}^2+\gamma_2 c^2\|\Delta \psi\|_{L^2}^2 +\gamma_3 \|\psi_{tt}\|_{L^2}^2\\
\leq&\,C(\gamma_1\|\psi_{tt}\|_{L^2}^2+(\gamma_2+\gamma_3) \|\nabla \psi_t\|_{L^2}^2+\|f\|_{L^2}^2).
\end{aligned}
\end{equation}
Using Poincar\'e's inequality and choosing 
\[
 \gamma_2 + \gamma_3 < b/C, \quad \gamma_1< \gamma_3/C,
 \]
 we obtain
\begin{equation}\label{dL_dt}
\ddt L(t)+ \|\nabla \psi_t\|_{L^2}^2+ \|\Delta \psi_t\|_{L^2}^2+ \|\nabla \psi\|_{L^2}^2+\|\Delta \psi\|_{L^2}^2+\|\psi_{tt}\|_{L^2}^2 \lesssim \|f\|_{L^2}^2.
\end{equation}   
Integrating \eqref{dL_dt} with respect to time and using equivalence \eqref{Eq_L_E} leads to
\begin{equation}\label{Energy_Estimate}
\sup_{t \in (0,T)} E(t)+\sup_{t \in (0,T)}  D(t)\lesssim  E(0)+ \int_0^t \|f(s)\|_{L^2}^2\ds.    
\end{equation}
Recalling the definition of $f$ in \eqref{f_form}, we have 
\begin{equation}
\begin{aligned}    
\int_0^t \|f(s)\|_{L^2}^2\ds\lesssim&\, \int_0^t\|\psi_t(s)\|_{L^\infty}^2\|\Delta \psi(s)\|_{L^2}^2 \ds+\int_0^t \|\nabla\psi(s)\|_{L^4}^2\|\nabla \psi_t(s)\|_{L^4}^2\ds\\
\lesssim&\,\int_0^t\|\Delta \psi_t(s)\|_{L^2}^2\|\Delta \psi(s)\|_{L^2}^2\ds\\
\lesssim&\, \sup_{t \in (0,T)} E(t) D(t). 
\end{aligned}
\end{equation}
Plugging this into \eqref{Energy_Estimate} yields
\begin{equation}\label{Energy_Estimate}
\sup_{t \in (0,T)} E(t)+D(t)\lesssim  E(0)+ \sup_{t \in (0,T)} E(t) D(t)
\end{equation}
 Hence, if $E(0)$ is small enough, a bootstrap argument leads to
 \begin{equation}
\sup_{t \in (0,T)} E(t)+D(t)\lesssim C.
\end{equation}
\indent We next prove the exponential decay of the energy.  Using \eqref{Agmon_Inequality}, we have by applying Agmon's and Young's inequalities,
\begin{equation}
\begin{aligned}
\|\psi_t\|_{L^\infty}\|\Delta \psi\|_{L^2} \leq&\, C_{\textup{A}}  \|\psi_t\|_{H^2(\Omega)}^{d/4}\|\psi_t\|_{L^2(\Omega)}^{1-d/4}\|\Delta \psi\|_{L^2} 
\\
\lesssim&\, \varepsilon\|\psi_t\|_{H^2(\Omega)}+C(\varepsilon)\left(\|\psi_t\|_{L^2(\Omega)}^{1-d/4}\|\Delta \psi\|_{L^2} \right)^{4/(4-d)}
\end{aligned}
\end{equation}
Applying Young's inequality yields 
\begin{equation}\label{Young_1}
\begin{aligned}
\|\psi_t\|_{L^\infty}^2\|\Delta \psi\|_{L^2}^2 \ds\leq&\, \varepsilon^2 \|\psi_t\|_{H^2(\Omega)}^2
+C(\varepsilon)(E(t))^{1+\kappa} \\ 
\end{aligned}    
\end{equation}
for some $\kappa>0$.  Similarly, we have by Lemma \ref{Lemma_Interpolation}
\begin{equation}\label{Young_2}
\begin{aligned}
\|\nabla\psi\|_{L^4}^2\|\nabla \psi_t\|_{L^4}^2\lesssim &\, \Vert \nabla \psi_t \Vert_{L^2}^{2(1-d/4)}\Vert   \psi_t\Vert_{H^2}^{d/2} \|\Delta \psi\|_{L^2}^2\\  
\leq&\, \varepsilon^2 \|\psi_t\|_{H^2(\Omega)}^2+C(\varepsilon)(E(t))^{1+\kappa} 
\end{aligned}
\end{equation}
Inserting \eqref{Young_1} and \eqref{Young_2} into \eqref{dL_dt}, and selecting $\varepsilon$ small enough leads to
\begin{equation}\label{L_D_Estimate_3}
  \ddt L(t)+ \|\nabla \psi_t\|_{L^2}^2+ \|\Delta \psi_t\|_{L^2}^2+ \|\nabla \psi\|_{L^2}^2+\|\Delta \psi\|_{L^2}^2+\|\psi_{tt}\|_{L^2}^2\lesssim (E(t))^{1+\kappa}. 
\end{equation}
From  the equivalence \eqref{Eq_L_E}, we deduce that there exists a positive constant $\zeta>0$, such that 
\begin{equation}\label{L_diff_Inq}
\ddt L(t)+\zeta L(t)\lesssim (L(t))^{1+\kappa}. 
\end{equation}
By integrating \eqref{L_diff_Inq} with respect to time, we obtain 
\begin{equation}
L(t)\leq c_1 e^{-\zeta t} L(0)+c_2\int_0^t e^{-\zeta(t-s)} (L(s))^{1+\kappa}\ds. 
\end{equation}
Applying Lemma \ref{Lemma_Expo} then with
 \begin{equation}
-\zeta+(1+1/\kappa)c_22^{\kappa}c_1^\kappa L(0)^{\kappa}<0
\end{equation}
gives 
\begin{equation}
L(t)\leq \left(1+\frac{c_2c_1^{\kappa}L(0)^\kappa}{-\zeta \kappa+(1+\kappa)c_2 2^{\kappa}c_1^{\kappa}L(0)^{\kappa}}\right)c_1 e^{-\zeta t}L(0). 
\end{equation}
Finally, employing the equivalence \eqref{Eq_L_E} yields the desired result. 
\end{proof} 
\begin{remark}[On the Kuznetsov equation]
Blackstock's equation can be viewed as an alternative model to the Kuznetsov equation~\cite{kuznetsov1971equations} given by
\begin{equation} \label{Kuznetsov}
	(1+2k \psi_t) \psi_{tt} -c^2 \D \psi -b \D \psi_t +2 \sigma \nabla \psi \cdot \nabla  \psi_t =0. 
\end{equation}
Although the developed theoretical framework can be transferred to \eqref{Kuznetsov} as well, we do not expect a gain in terms of the regularity assumptions compared to the available results in the literature in~\cite{mizohata1993global, kawashima1992global}. The reason is that the right-hand side nonlinearity $f$  in \eqref{f_form} would contain $\psi_t \psi_{tt}$. Then $\|f\|_{L^2 L^2}$ would involve $\|\psi_t \psi_{tt}\|_{L^2L^2}$, which cannot be controlled by $E(t)D(t)$ in their present form.     
 Therefore,  having a higher-order energy functional and assuming $(\psi_0, \psi_1) \in H^3(\Omega)\times H^2(\Omega)$ in the global well-posedness analysis of \eqref{Kuznetsov} seems  necessary within the present framework. We note, however, that \eqref{Kuznetsov} also appears in the pressure (or pressure-velocity) form in the literature, which allows for weaker regularity assumptions on the data; see~\cite{meyer2012global, kaltenbacher2011well, kaltenbacher2012analysis}.   
\end{remark}
 \bibliography{references}{}   
\bibliographystyle{siam}   
  \end{document}

%% file: preamble.tex
\makeatletter
\@namedef{subjclassname@2020}{\textup{2020} Mathematics Subject Classification}
\makeatother

\usepackage{lipsum}
\usepackage{amsfonts}
\usepackage{graphicx}
\usepackage{epstopdf}
\usepackage{algorithmic}
\usepackage{calligra}
\usepackage[dvipsnames]{xcolor}
\usepackage{amsfonts,amsmath,amsthm,amssymb}
\usepackage{mathtools}
\usepackage{hyperref}
\usepackage[makeroom]{cancel}
\usepackage{autonum}
\usepackage{hhline}
\usepackage{array}
\usepackage{diagbox}
\usepackage{tcolorbox}
\usepackage{mdframed}
\usepackage{multicol}
\usepackage{graphicx}
\usepackage{subcaption}
\usepackage{moreverb}
\usepackage{bbm}
\usepackage[margin=1.38in]{geometry}
\usepackage{todonotes}
\usepackage{scalerel,amssymb}
\allowdisplaybreaks
\usepackage{mathrsfs}  
\usepackage{lineno}
\usepackage{todonotes}
\usepackage{tikz}
\usepackage{appendix}
\usepackage{enumitem}
\usepackage{pgfplots}
\usetikzlibrary{arrows.meta}
\usepackage{siunitx}
\usepackage[numbers,sort&compress]{natbib}
\definecolor{mygreen}{HTML}{43a047}
\usepackage{subcaption}
\usepackage{doi}
\usepackage{soul}

\newcommand{\D}{\Delta}

\newcommand{\ddt}{\frac{\textup{d}}{\textup{d}t}}
\newcommand{\dds}{\frac{\textup{d}}{\textup{d}s}}

\newcommand{\dt}{\, \textup{d} t}
\newcommand{\ds}{\, \textup{d} s }
\newcommand{\dx}{\, \textup{d} x}
\newcommand{\dxs}{\, \textup{d}x\textup{d}s}

\newcommand{\intO}{\int_{\Omega}}

\newcommand{\nLtwoLtwo}[1]{\|#1\|_{L^2 (L^2)}}
\newcommand{\nLtwoLfour}[1]{\|#1\|_{L^2 (L^4)}}

\newcommand{\nLtwoLinf}[1]{\|#1\|_{L^2 (L^\infty)}}
\newcommand{\nLinfLtwo}[1]{\|#1\|_{L^\infty (L^2)}}
\newcommand{\nLinfLr}[1]{\|#1\|_{L^\infty (L^r)}}

\newcommand{\nLinfLfour}[1]{\|#1\|_{L^\infty (L^4)}}

\newcommand{\nLtwoHneg}[1]{\|#1\|_{L^2 (H^{-1})}}

\newcommand{\R}{\mathbb{R}} 
 
\newcommand{\Ltwo}{L^2(\Omega)}

\newcommand{\Hone}{H^1(\Omega)}

\newcommand{\Honezero}{H_0^1(\Omega)}

\newcommand{\Honetwo}{{H^2(\Omega)\cap H_0^1(\Omega)}}

\newcommand{\Xv}{\mathcal{X}^v}
\newcommand{\Xtv}{\mathcal{X}^v_t}
\newcommand{\Xt}{\mathcal{X}^\psi_t}
\newcommand{\Xpsi}{\mathcal{X}^\psi}

\newtheorem{theorem}{Theorem}
\newtheorem{lemma}{Lemma}
\newtheorem{proposition}{Proposition}

\newtheorem{remark}{Remark}
\numberwithin{lemma}{section}
\numberwithin{proposition}{section}
\numberwithin{theorem}{section}
\numberwithin{equation}{section}
\makeatletter
\newcommand{\leqnomode}{\tagsleft@true}
\newcommand{\reqnomode}{\tagsleft@false}
\makeatother

\definecolor{grey}{rgb}{0.5,0.5,0.5}
 